\documentclass[11pt]{article}
\usepackage[T1]{fontenc}
\usepackage{lmodern}
\usepackage[headings]{fullpage}               
\usepackage{boxedminipage}
\usepackage{amsfonts}
\usepackage{amsmath} 
\usepackage{amssymb}
\usepackage{graphicx}
\usepackage{amsthm}
\usepackage{subfig}
\usepackage{tikz}
\usepackage{setspace}
\usepackage[colorlinks=true,linkcolor=blue,urlcolor=blue,citecolor=red]{hyperref}
\usepackage[nameinlink]{cleveref}
\usepackage{enumerate}
\usepackage{enumitem}
\usepackage{optidef}
\hypersetup{linktocpage=true,}
\newtheorem{theorem}{Theorem}[section]
\newtheorem{lemma}[theorem]{Lemma}

\newtheorem{proposition}[theorem]{Proposition}
\newtheorem{corollary}[theorem]{Corollary}

\theoremstyle{definition}\newtheorem{definition}[theorem]{Definition}
\theoremstyle{definition}\newtheorem{example}[theorem]{Example}
\theoremstyle{definition}\newtheorem{remark}[theorem]{Remark}

\newcommand{\rank}{\operatorname{rank}}

\newcommand{\conv}{\operatorname{conv}}
\newcommand{\esssup}{\operatorname{ess}\operatorname{sup}}

\tikzstyle{vertex}=[circle, draw, fill=black, inner sep=0pt, minimum size=4pt]
\tikzstyle{blankvertex}=[circle, draw=white, fill=white, inner sep=0pt, minimum size=4pt]
\tikzstyle{edge}=[line width=1.5pt]
\tikzstyle{labelsty}=[font=\scriptsize]

\begin{document}

\title{Edge-length preserving embeddings of graphs between normed spaces}

\author{Sean Dewar\thanks{School of Mathematics, University of Bristol, Bristol, BS8 1UG, UK. \texttt{sean.dewar@bristol.ac.uk}} \qquad 
Eleftherios Kastis\thanks{Mathematics and Statistics, Lancaster University, Lancaster, LA1 4YF, UK. {\tt l.kastis@lancaster.ac.uk}} \qquad  
Derek Kitson\thanks{Mathematics and Computer Studies, Mary Immaculate College, Limerick, V94 VN26, Ireland. {\tt Derek.Kitson@mic.ul.ie}}  \qquad 
William Sims\thanks{Computer \& Information Science \& Engineering, University of Florida, Gainesville, FL 32611, USA. \texttt{w.sims@ufl.edu}}}

\date{}
\maketitle

\begin{abstract}
The concept of graph flattenability, initially formalized by Belk and Connelly and later expanded by Sitharam and Willoughby, 
extends the question of embedding finite metric spaces into a given normed space. A finite simple graph $G=(V,E)$ is said to be $(X,Y)$-flattenable if any set of induced edge lengths from an embedding of $G$ into a normed space $Y$ can also be realised by an embedding of $G$ into a normed space $X$. This property, being minor-closed, can be characterized by a finite list of forbidden minors. Following the establishment of fundamental results about $(X,Y)$-flattenability, we identify sufficient conditions under which it implies independence with respect to the associated rigidity matroids for $X$ and $Y$. We show that the spaces $\ell_2$ and $\ell_\infty$ serve as two natural extreme spaces of flattenability and discuss $(X, \ell_p )$-flattenability for varying $p$. We provide a complete characterization of $(X,Y)$-flattenable graphs for the specific case when $X$ is 2-dimensional and $Y$ is infinite-dimensional.

\end{abstract}



\section{Introduction}

A \emph{realisation} of a (finite simple) graph $G=(V,E)$, with at least one edge, in a real normed linear space $(X,\|\cdot\|_X)$ (referred to simply as $X$ when the context is clear) is a map $p:V \rightarrow X$.
The \emph{measurement map} of the pair $(G,X)$ is the map
\begin{align*}
	f_{G,X} : X^V \rightarrow \mathbb{R}^E, ~ (x_v)_{v \in V} \mapsto \left( \| x_v-x_w \|_X \right)_{vw \in E},
\end{align*}
which sends each realisation to its corresponding vector of induced edge lengths.
Given another normed space $Y$,
we say that the graph $G$ can be \emph{flattened from $Y$ into $X$} (or more succinctly, $G$ is \emph{$(X,Y)$-flattenable}) if every vector of edge lengths induced by a realisation in $Y$ can also be induced by a realisation in $X$.
Equivalently,
a graph is $(X,Y)$-flattenable if and only if $f_{G,X}(X^V) \supseteq f_{G,Y}(Y^V)$.  
For example,
it is well-known that any $n$ points in $\ell_2$ can be isometrically embedded into $\ell_2^{n-1}$.
Within our framework of graph flattenability, this says that any graph with $n$ vertices must be $(\ell_2^{n-1},\ell_2)$-flattenable.

 The concept of flattenability was first formalized by Belk and Connelly \cite{belkconn07} when they characterized the set of $(\ell_2^d, \ell_2)$-flattenable graphs, referring to them as ``realisable'', for any $d \leq 3$.  
Sitharam and Willoughby \cite{SithWill} changed the name to ``flattenability'' and extended its definition to prove several results about the set of $(\ell^d_p, \ell_p)$-flattenable graphs, for all $d \geq 1$ and all $p \in [1,\infty]$.  
Notably, they connected the work in \cite{belkconn07} to that of Ball \cite{ball90} and Witsenhausen \cite{witt}, and to the area of metric space embeddings in general.

A \emph{minor} of a graph $G$ is any graph obtained via a sequence of edge deletions and edge contractions - i.e., deleting two vertices connected by an edge and adding a vertex whose neighborhood is the union of the neighborhoods of the deleted vertices.  
It is easy to see that if $G$ is $(X,Y)$-flattenable, then so are all of its minors.  
Hence, the famous Robertson–Seymour theorem \cite{rs04} shows that, for every ordered pair $(X,Y)$ of normed spaces, there is a finite list of forbidden minors that characterise the $(X,Y)$-flattenable graphs.
If these forbidden minors are known, then $(X,Y)$-flattenability can be determined in polynomial time \cite{rs95}.  

Belk and Connelly \cite{belkconn07} showed that the forbidden minors for the set of $(\ell^d_2,\ell_2)$-flattenable graphs are $K_3$ for $d=1$, $K_4$ for $d=2$, and $K_5$ and $K_{2,2,2}$ for $d = 3$ (see \Cref{fig:k222}).  
The forbidden minors for $(\ell^d_2,\ell_2)$-flattenability are unknown for all $d \geq 4$, but it is known that they must be a subset of the forbidden minors for a class of graphs called \emph{partial $d$-trees}.  
Resolving a conjecture posed in \cite{SithWill}, Fiorini et al.~\cite{fhjv2017} proved that the forbidden minors for the set of $(\ell^2_{\infty},\ell_{\infty})$-flattenable graphs (and also for the $(\ell^2_1,\ell_1)$-flattenable graphs) are $W_4$ and $K_4 +_e K_4$ (see \Cref{fig:w5join}).

Until now, flattenability research has primarily focused on determining the lowest dimension of an $\ell_p$-space into which a graph can be flattened. These known results are collected in \Cref{sec:lpknown}. In this paper we investigate flattenability between general normed spaces $X$ and $Y$. 
Basic results concerning $(X,Y)$-flattenability are contained in \Cref{sec:general}, including a full characterisation of $(X,Y)$-flattenable graphs when either $X$ or $Y$ is the real line. 

In \Cref{sec:ind}, we generalize the result of Sitharam and Willoughby \cite{SithWill}, identifying sufficient conditions under which $(X,Y)$-flattenability implies that the graph is \emph{independent in $X$}, in that the the rigidity map $f_{G,X}$ has a differentiable point where the Jacobian has rank $|E|$ (with the assumption that $X$ is finite-dimensional).  Namely, we show that if a graph $G$ is $(X,Y)$-flattenable with $X$ being finite-dimensional normed space, and if either $G$ is independent in $Y$ or $Y$ is infinite-dimensional, then $G$ is independent in $X$, provided that 
the norm $\|\cdot\|_X$ satisfies a mild smoothness condition.

In \Cref{sec:lp}, we highlight that the spaces $\ell_2$ and $\ell_\infty$ serve as two natural extreme spaces of flattenability (\Cref{t:linfl2}). 
Since every finite metric space can be embedded in $\ell_\infty$, every $(X,\ell_\infty)$-flattenable graph is $(X,Y)$-flattenable for any normed space $Y$. Conversely, if a graph $G$ is $(X,Y)$-flattenable, where $X$ is finite-dimensional and $Y$ is infinite-dimensional, then $G$ is also $(X,\ell_2)$-flattenable. 
In \Cref{t:pqembed2}, we show that the set of $p$-values such that $G$ is $(X,\ell_p)$-flattenable is closed, and is either empty or contains an interval of the form $[q,2]$, where $1\leq q\leq 2$. 

Forbidden minors for flattenability into normed planes are discussed in \Cref{sec:normedplane}.
Given $X$ is a normed plane and $Y$ is a normed space with dimension 3 or more,
we prove that $(X,Y)$-flattenability is split into two cases depending on whether or not $X$ is isometrically isomorphic to $\ell_\infty^2$ (\Cref{t:2dcombined}, with \Cref{t:2d1,t:l2tolinf} respectively for the individual cases).
These results can all be extended to allow for $X$ to have dimension 3 or more so long as $X$ is strictly convex and $Y$ is not (\Cref{t:strconv}).
\Cref{t:2d1} can also be improved when $X = \ell_2^2$, where we are able to fully characterise $(X,Y)$-flattenability for all choices of $Y$ (\Cref{t:ell2}).

In the last section (\Cref{sec:infinite}), we extend our definition of flattenability for countable simple graphs. In \Cref{t:infinite}, we prove that a countable graph is $(X,Y)$-flattenable to a finite-dimensional space $X$ if and only if it contains a complete tower of connected $(X,Y)$-flattenable subgraphs.

\section{Previously known results in \texorpdfstring{$\ell_p$}{lp} spaces}\label{sec:lpknown}
In this section, we present previously known results for flattenability in $\ell_p$ spaces.  First, we formally define the spaces $\ell_p$ and $\ell^d_p$.  For each $p \in [1,\infty)$ and index set $I$, define the linear spaces 
\begin{align*}
 \ell_p(I)= \left\{(x_i)_{i\in I} \in \mathbb{R}^I : \sum_{i\in I} |x_i|^p < \infty\right\} 
 \text{ and }  \ell_\infty(I)= \left\{(x_i)_{i\in I} \in \mathbb{R}^I : \sup_{i\in I} |x_i| < \infty\right\}. 
\end{align*}
and endow them with the respective norms
\begin{align*}
\|x\|_p := \left( \sum_{i \in I} |x_n|^p\right)^{1/p} \text{ and } \|x\|_\infty= \sup_{i\in I} |x_i|.
\end{align*}
It is well known that for every $p\in[1,\infty]$, the space $\ell_p(I)$ is complete with respect to the metric induced by its norm, so it is a Banach space. We denote the finite-dimensional normed space $\ell_p(\{1,\dots,d\})$ as $\ell_p^d$. This space corresponds to the space $\mathbb{R}^d$ with the $\|\cdot\|_p$ norm. We also denote the sequence space $\ell_p(\mathbb{N})$ as $\ell_p$.

\subsection{Flattenability for Complete Graphs}
Here we translate previous results about isometric embeddings of point sets into the language of $(\ell_p^d,\ell_p)$-flattenability, for various values of $p$ and $d$.  
In particular, \Cref{t:l1}, \Cref{t:lp}, and \Cref{t:linf}, below, provide lower-bounds on $d$ for which a given graph is $(\ell_p^d,\ell_p)$-flattenable, and upper-bounds on $d$ for which complete graphs on sufficiently many vertices are not $(\ell_p^d,\ell_p)$-flattenable.  

\begin{theorem}[\cite{ball90,witt}]\label{t:l1}
	Let $G$ be a graph with $n$ vertices.
	\begin{enumerate}
	    \item $G$ is $(\ell_1^d,\ell_1)$-flattenable for each $d \geq \binom{n}{2}$.
	    \item If $n \geq 4$ and $G$ is complete, then $G$ is not $(\ell_1^d,\ell_1)$-flattenable for each $d < \binom{n-2}{2}$.
	\end{enumerate}
\end{theorem}

\begin{theorem}[\cite{ball90}]\label{t:lp}
	Let $G$ be a graph with $n$ vertices and let $p \in (1,\infty)$.
	\begin{enumerate}
	    \item $G$ is $(\ell_p^d,\ell_p)$-flattenable for each $d \geq \binom{n}{2}$.
	    \item If $n \geq 3$, $p < 2$ and $G$ is complete,
	    then $G$ is not $(\ell_p^d,\ell_p)$-flattenable for each $d < \binom{n-1}{2}$.
	\end{enumerate}
\end{theorem}

\begin{theorem}[\cite{hol}]\label{t:linf}
	Let $G$ be a graph with $n$ vertices.
	\begin{enumerate}
	    \item If $n \geq 4$, then $G$ is $(\ell_\infty^d,\ell_\infty)$-flattenable for each $d \geq n-2$.
	    \item If $G$ is complete, then $G$ is not $(\ell_\infty^d,\ell_\infty)$-flattenable for each $d < \lfloor 2n/3 \rfloor$.
	\end{enumerate}
\end{theorem}

Although the upper and lower bounds in \Cref{t:l1,t:lp} are not tight for any complete graph,
the bounds described in \Cref{t:linf} are tight for sufficiently small values of $n$.
In particular, the upper and lower bounds tell us that the complete graph $K_n$ is $(\ell_\infty^{n-2},\ell_\infty)$-flattenable but not $(\ell_\infty^{n-3},\ell_\infty)$-flattenable for each $n \in \{4,5,6\}$.  
In \cite{fhjv2017}, it was shown that this also holds for $n=7$; specifically, the complete graph $K_7$ is $(\ell_\infty^5,\ell_\infty)$-flattenable but not $(\ell_\infty^4,\ell_\infty)$-flattenable.

The lower bound in \Cref{t:linf} can be improved asymptotically.
In \cite{ball90},
Ball proved that if the complete graph on $n$ vertices is $(\ell_\infty^d,\ell_\infty)$-flattenable,
then every graph with $n$ vertices can be covered by $d$ complete bipartite graphs.
Using this,
he proved that there exists a constant $c>0$ such that each complete graph $K_n$ is not $(\ell_\infty^d,\ell_\infty)$-flattenable for each $d < n - c n^{3/4}$.
Fiorini et al.~\cite{fhjv2017} later observed that this can be improved by using more recent results regarding bipartite graph coverings \cite{rodl}.

\begin{theorem}[\cite{ball90,rodl}]\label{t:linfasymm}
	There exists a constant $c>0$ such that each complete graph $K_n$ is not $(\ell_\infty^d,\ell_\infty)$-flattenable for each $d < n - c \log n$.
\end{theorem}

In \cite{ball90}, Ball also noted that his method for constructing asymptotic lower bounds cannot be improved past $n - c \log n$, since there exists $c' >0$ such that every graph with $n$ vertices can be covered by $n-c'\log n$ complete bipartite graphs.
This latter statement follows from the Ramsey number $R(n)$ being at most $4^n$.

\subsection{Forbidden minor characterisations for flattenability}

In this subsection, we present previously known finite forbidden minor characterisations for flattenability.  
In the Euclidean case, we highlight the results of Belk and Connelly in \Cref{t:2dflat} and \Cref{t:3dflat}, below.

\begin{theorem}[\cite{belkconn07}]\label{t:2dflat}
	For $d \in \{1,2\}$, a graph $G$ is $(\ell_2^d,\ell_2)$-flattenable if and only if it contains no $K_{d+2}$ minor.
\end{theorem}

\begin{theorem}[\cite{belkconn07}]\label{t:3dflat}
	A graph $G$ is $(\ell_2^3,\ell_2)$-flattenable if and only if it contains no $K_5$ or $K_{2,2,2}$ minor (see \Cref{fig:k222}).
\end{theorem}

\begin{figure}[htp]
\begin{center}
\begin{tikzpicture}[scale=1.1]
        \node[vertex] (a) at (0,2) {};
        \node[vertex] (b) at (-1,1.2) {};
        \node[vertex] (c) at (-0.7,0) {};	
        \node[vertex] (d) at (0.7,0) {};
   	\node[vertex] (e) at (1,1.2) {};
    	
    	\draw[edge] (a)edge(b);
    	\draw[edge] (a)edge(c);
    	\draw[edge] (a)edge(d);
    	\draw[edge] (a)edge(e);
    	\draw[edge] (b)edge(c);
    	\draw[edge] (b)edge(d);
    	\draw[edge] (b)edge(e);
    	\draw[edge] (c)edge(d);
    	\draw[edge] (c)edge(e);
    	\draw[edge] (d)edge(e);
    \end{tikzpicture} \qquad \qquad
    \begin{tikzpicture}[scale=0.8]
        \node[vertex] (a) at (0,2) {};
    	\node[vertex] (b) at (1.732,-1) {};
        \node[vertex] (c) at (-1.732,-1) {};
    	
        \node[vertex] (1) at (0,-0.5) {};
    	\node[vertex] (2) at (-0.5,0.433) {};
        \node[vertex] (3) at (0.5,0.433) {};
    	
    	\draw[edge] (a)edge(b);
    	\draw[edge] (b)edge(c);
    	\draw[edge] (c)edge(a);
    	\draw[edge] (1)edge(2);
    	\draw[edge] (2)edge(3);
    	\draw[edge] (3)edge(1);
    	\draw[edge] (1)edge(b);
    	\draw[edge] (1)edge(c);
    	\draw[edge] (2)edge(a);
    	\draw[edge] (2)edge(c);
    	\draw[edge] (3)edge(a);
    	\draw[edge] (3)edge(b);
    \end{tikzpicture}
\end{center}
\caption{The complete graph $K_5$ (left) and the complete tripartite graph $K_{2,2,2}$ (right).}
\label{fig:k222}
\end{figure}
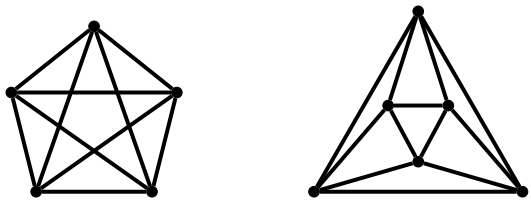

Witsenhausen showed in \cite{witt} that $K_4$ is  $(\ell_\infty^2,\ell_\infty)$-flattenable.  Sitharam and Willoughby \cite{SithWill} used convexity arguments on the so-called Cayley configuration space over specified non-edges of a $d$-dimensional framework to show that $K_5$ minus an edge is not $(\ell_\infty^2,\ell_\infty)$-flattenable.
From this, they conjectured that the wheel graph $W_4$ is the only forbidden minor for the class of $(\ell_\infty^2,\ell_\infty)$-flattenable graphs on at most $5$ vertices. 
Fiorini et al.~\cite{fhjv2017} verified this conjecture, determined the complete set of forbidden minors for $(\ell_{\infty}^2,\ell_{\infty})$-flattenability, and showed that these minors also completely characterise $(\ell_1^2,\ell_1)$-flattenability.  
We state these results in \Cref{t:infflat}.

\begin{theorem}[\cite{fhjv2017}]\label{t:infflat}
	For any $p \in \{1,\infty\}$, a graph $G$ is $(\ell_p^2,\ell_p)$-flattenable if and only if it contains no $W_4$ or $K_4 +_e K_4$ minor (see \Cref{fig:w5join}).
\end{theorem}

\begin{figure}[htp]
\begin{center}
    \begin{tikzpicture}[scale=1]
        \node[vertex] (a) at (2,0) {};
    	\node[vertex] (b) at (1,-1) {};
        \node[vertex] (c) at (1,1) {};
    	\node[vertex] (d) at (0,0) {};
    	
        \node[vertex] (e) at (1,0) {};
    	
    	\draw[edge] (e)edge(a);
    	\draw[edge] (e)edge(b);
    	\draw[edge] (e)edge(c);
    	\draw[edge] (e)edge(d);
    	\draw[edge] (a)edge(b);
    	\draw[edge] (a)edge(c);
    	\draw[edge] (b)edge(d);
    	\draw[edge] (c)edge(d);
    	\filldraw[black] (0,-1.5) circle {};
    \end{tikzpicture}\qquad \qquad
    \begin{tikzpicture}[scale=0.75]
        \node[vertex] (a) at (-2,-1) {};
    	\node[vertex] (b) at (-2,1) {};
    	
        \node[vertex] (c) at (0,1) {};
    	\node[vertex] (d) at (0,-1) {};
    	
        \node[vertex] (e) at (2,-1) {};
    	\node[vertex] (f) at (2,1) {};
    	
    	\node[blankvertex] (blank) at (0,-2) {};
    	
    	\draw[edge] (a)edge(b);
    	\draw[edge] (a)edge(c);
    	\draw[edge] (a)edge(d);
    	\draw[edge] (b)edge(c);
    	\draw[edge] (b)edge(d);
        
    	\draw[edge] (c)edge(e);
    	\draw[edge] (c)edge(f);
    	\draw[edge] (d)edge(e);
    	\draw[edge] (d)edge(f);
    	\draw[edge] (e)edge(f);
    \end{tikzpicture}
\end{center}
\vspace{-0.5cm}
\caption{The wheel graph $W_4$ (left) and the graph $K_4 +_e K_4$ formed by joining two copies of $K_4$ at an edge $e$ and then removing said edge (right).}
\label{fig:w5join}
\end{figure}
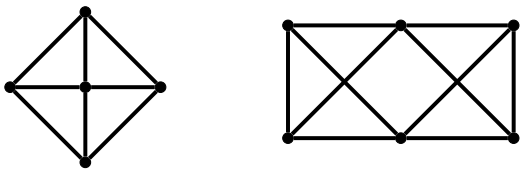

Very little is known regarding the forbidden minors for $(\ell_1^d,\ell_1)$-flattenability and $(\ell_\infty^d,\ell_\infty)$-flattenability when $d \geq 3$.
Some families of forbidden minors for $(\ell_\infty^d,\ell_\infty)$-flattenability can be found in \cite{fhjm}.

\section{Basic results for flattenability between general normed spaces}\label{sec:general}

In this section we cover general properties of flattenability between general normed spaces.
Throughout the paper we shall make the (rather trivial) assumption that every normed space has dimension higher than zero.
We now begin with some easy observations.

\begin{proposition}\label{p:xyzembed}
    If a graph $G$ is $(X,Y)$-flattenable and $(Y,Z)$-flattenable,
    then $G$ is $(X,Z)$-flattenable.
\end{proposition}

\begin{proof}
    This is immediate as $f_{G,X}(X^V) \supseteq f_{G,Y}(Y^V) \supseteq f_{G,Z}(Z^V)$.
\end{proof}

\begin{proposition}\label{p:isom}
	Let $X,Y$ be isometrically isomorphic normed spaces.
	Then the following holds for every normed space $Z$.
	\begin{enumerate}
		\item\label{p:isom1} Every $(X,Z)$-flattenable graph is $(Y,Z)$-flattenable.
		\item\label{p:isom2} Every $(Z,X)$-flattenable graph is $(Z,Y)$-flattenable.
	\end{enumerate}
\end{proposition}

\begin{proof}
	Let $G=(V,E)$ be any graph and let $T:X \rightarrow Y$ be an isometric linear isomorphism between $X$ and $Y$.
	
	\ref{p:isom1}:
	Suppose $G$ is $(X,Z)$-flattenable graph and choose any $q \in Z^V$.
	Then there exists $p \in X^V$ so that $f_{G,X}(p)=f_{G,Z}(q)$.
	If we define $p' \in Y^V$ with $p'_v:=T(p_v)$ for each $v \in V$,
	then $f_{G,Y}(p') = f_{G,X}(p)=f_{G,Z}(q)$ as required.
	
	\ref{p:isom2}:
	Suppose $G$ is $(Z,X)$-flattenable graph and choose any $q \in Y^V$.
	If we define $q' \in X^V$ with $q'_v := T^{-1}(q_v)$ for each $v \in V$,
	then $f_{G,X}(q') = f_{G,Y}(q)$.
	As $G$ is $(Z,X)$-flattenable then there exists $p \in Z^V$ with $f_{G,Z}(p) = f_{G,X}(q') = f_{G,Y}(q)$ as required.
\end{proof}

We can also immediately characterise the $(X,Y)$-flattenable graphs when either $\dim X = 1$ or $\dim Y = 1$.

\begin{proposition}\label{p:isomembed}
    If $Y$ can be isometrically embedded into $X$, then every graph is $(X,Y)$-flattenable.
    In particular, every graph is $(X,Y)$-flattenable when $\dim Y = 1$.
\end{proposition}

\begin{proof}
    Let $G=(V,E)$ be any graph and let $T:Y \rightarrow X$ be an isometric linear map.
    As $f_{G,Y}(Y^V) = f_{G,X}((TY)^V) \subseteq f_{G,X}(X^V)$,
    $G$ is $(X,Y)$-flattenable.
    The final part of the result now follows from the observation that all 1-dimensional normed spaces are isometrically isomorphic,
    and hence any 1-dimensional normed space can be isometrically embedded into any higher dimensional space.
\end{proof}

\begin{proposition}\label{p:1dflat}
	Let $X,Y$ be normed spaces where $\dim X =1$.
	\begin{enumerate}
	    \item \label{p:1dflat1} If $\dim Y = 1$, then every graph is $(X,Y)$-flattenable.
	    \item \label{p:1dflat2} If $\dim Y \geq 2$, then a graph is $(X,Y)$-flattenable if and only if it is a forest.
	\end{enumerate}
\end{proposition}

\begin{proof}
	\ref{p:1dflat1}:
	This follows from \Cref{p:isomembed}.
	
	\ref{p:1dflat2}:
	Choose any graph $G=(V,E)$.
	It is clear that if $G$ is a forest with specified edge lengths $(d_{vw})_{vw \in E}$,
	then there exists $p \in X^V$ where $\|p_v-p_w\|_X =d_{vw}$ for all $vw \in E$.
	Suppose $G$ is not a forest.
	Since $(X,Y)$-flattenability is a minor-closed property,
	it suffices to assume that $G$ is a complete graph with three vertices, i.e., $G \cong K_3$.
	It is easy to see that for any 2-dimensional subspace $Z \subset Y$,
	we can choose three points $x,y,z \in Z$ where $\|x-z\|_Y=\|y-z\|_Y=\|x-y\|_Y=1$.
	We note $f_{G,Y}(x,y,z)=(1,1,1)$,
	however there exists no $p \in X^V$ with $f_{G,X}(p)=(1,1,1)$.
\end{proof}

\section{Independence and flattenability}\label{sec:ind}

In this section we explore connections between $(X,Y)$-flattenability and independence, a matroidal property found in rigidity theory. There are two main results which deal separately with the case where $Y$ is finite dimensional (\Cref{t:ind}) and where $Y$ is infinite dimensional (\Cref{c:ind}).

A realisation $p$ of a graph $G=(V,E)$ in a finite-dimensional normed space $X$ is \emph{independent} if the map $f_{G,X}$ is (Fr\'{e}chet) differentiable at $p$ and $\rank df_{G,X}(p) = |E|$. 
If an independent realisation of a graph $G$ exists in a normed space $X$ then we say that $G$ is \emph{independent in $X$}.
It was proven in \cite{SithWill} that any $(\ell_p^d,\ell_p)$-flattenable graph is also independent in $\ell_p^d$.
In this section we similarly prove that this extends to most normed spaces.

We now require the following concept.
Let $X,Y$ be finite-dimensional normed spaces and let $f :X \rightarrow Y$ be a locally Lipschitz map,
i.e., a map such that for every point $x \in X$,
there exists an open neighbourhood $U_x \subset X$ of $x$ and a constant $k_x >0$ such that $\|f(x) - f(x')\|_Y \leq k_x \|x-x'\|_X$ for all $x' \in U_x$.
Let $D(f)$ denote the set of differentiable points of a locally Lipschitz map $f:X\to Y$.
By Rademacher's theorem,
 $D(f)$ is a conull set 
and hence  is dense in $X$.
Let $x\in X$ and denote by $D(f;x)$ the set of sequences $(x_n)$ in $D(f)$ which converge to $x$. 

We define the set
\begin{equation*}
    \partial f(x) := \conv \left\{ \lim_{n \rightarrow \infty} df(x_n) : (x_n)\in D(f;x) \text{ and } (df(x_n)) \text{ converges }  \right\},
\end{equation*}
where $\conv S$ denotes the convex hull of a set $S$.
Any linear map in $\partial f(x)$ is called a \emph{generalised derivative} of $f$ at $x$.
Each set $\partial f(x)$ is non-empty, convex and compact \cite[Proposition 2.6.2(a)]{clarke}.
Sets of generalised derivatives also obey the following continuity rule.

\begin{lemma}[{\cite[Proposition 2.6.2(c)]{clarke}}]\label{l:clarkecont}
    Let $X,Y$ be finite-dimensional normed spaces, let $f:X \rightarrow Y$ be a locally Lipschitz map, and let $x_0 \in X$.
    Then for every $\varepsilon >0$,
    there exists $\delta >0$ such that for each $x \in X$ with $\|x-x_0\|_X < \delta$,
    we have
    \begin{equation*}
        \partial f(x) \subset \partial f(x_0) + B_\varepsilon,
    \end{equation*}
    where $B_\varepsilon$ is the set of all linear maps $T:X \rightarrow Y$ with $\|T\|_{\text{op}} < \varepsilon$ (with $\|\cdot\|_{\text{op}}$ being the operator norm for linear maps between $X$ and $Y$).
\end{lemma}

Generalised derivatives allow for a non-smooth variant of the constant rank theorem.

\begin{theorem}[{\cite[Theorem 3.1]{btv}}]\label{t:rankthm}
    Let $X,Y$ be finite-dimensional normed spaces, let $f:X \rightarrow Y$ be a locally Lipschitz map, and let $x \in X$.
    Suppose that there exists a neighbourhood $O \subset X$ of $x$ such that every generalised derivative of every point $x' \in O$ has rank $k$.
    Then there exists open sets $U \subset \mathbb{R}^{\dim X}$, $V \subset \mathbb{R}^{\dim Y}$, $U' \subset O$, $V' \subset Y$ with $x \in U'$, $f(x) \in V'$,
    and there exist bilipschitz maps $\phi:U \rightarrow U'$ and $\psi: V \rightarrow V'$,
    so that for every $(a_1,\ldots,a_{\dim X})\in U$ we have 
    \begin{equation*}
        (\psi^{-1} \circ f \circ \phi) (a_1,\ldots,a_{\dim X}) = (a_1,\ldots,a_k, 0, \ldots,0).
    \end{equation*}
\end{theorem}

\begin{lemma}\label{l:openint}
    If a graph $G$ is independent in a finite-dimensional normed space $X$, then the set $f_{G,X}(X^V)$ has a non-empty interior.
\end{lemma}

\begin{proof}
    Choose a realisation $p$ of $G$ in $X$ such that $f_{G,X}$ is differentiable at the point $p$ and $\rank d f_{G,X}(p) =|E|$.
    The projection of the map $f_{G,X}$ to any coordinate in its codomain describes a convex map from $X^V$ to $\mathbb{R}$.
    Since the derivative of a convex function is continuous over its set of differentiable points (see, for example, \cite[Theorem 25.5]{rockafellar}),
    it follows that $\partial f_{G,X}(p) = \{df_{G,X}(p)\}$.
    By \Cref{l:clarkecont},
    there exists an open neighbourhood $O$ of $p$ such that for each $q \in O$,
    every generalised derivative of $f_{G,X}$ at $q$ has rank $|E|$.
    Since projections are open maps, it follows now from \Cref{t:rankthm} that the set $f_{G,X}(O)$ has a non-empty interior.
\end{proof}

Before moving to our first main result, we describe the following concepts from differential geometry.
Let $X$ and $Y$ be finite-dimensional normed spaces, $U \subset X$ be an open set, and $f:U \subset X \rightarrow Y$ a $C^\infty$-differentiable map (i.e., for each positive integer $k$, the $k$-th Fr\'{e}chet derivative of $f$ exists and is continuous).
We say that a point $y \in Y$ is a \emph{critical value} if there exists a point $x \in X$ where $f(x)=y$ and $\rank df(x) < \dim Y$. 
Any point in $Y$ that is not a critical value is said to be a \emph{regular value}.

\begin{theorem}\label{t:ind}
    Let $X$ be a finite-dimensional normed space where the norm is $C^\infty$-differentiable on an open dense subset of $X$.
    If a graph $G$ is independent in a finite-dimensional normed space $Y$ and $(X,Y)$-flattenable,
    then $G$ is independent in $X$.
\end{theorem}

\begin{proof}
    By \Cref{l:openint}, $f_{G,Y}(Y^V)$ has a non-empty interior.
    As $G$ is $(X,Y)$-flattenable,
    it follows that $f_{G,X}(X^V)$ also has a non-empty interior.
    Let $U \subset X^V$ be an open dense set of realisations of $G$ in $X$ where $f_{G,X}$ is $C^\infty$-differentiable.
    By Sard's theorem,
    the set of regular values of the restricted map $f_{G,X}:U\to \mathbb{R}^E$ is a dense subset of $\mathbb{R}^E$.
    It follows that the set $f_{G,X}(X^V)$ contains a regular value as it has a non-empty interior.
    Hence there exists $p$ in $U$ where $\rank df_{G,X}(p)=|E|$,
    and thus $G$ is independent in $X$.
\end{proof}

To prove our second main result we will require the following  concepts.
Let $p$ be a realisation of a graph $G=(V,E)$ in a finite-dimensional normed space $X$ such that the measurement map $f_{G,X}$ is differentiable at $p$.
For any non-zero point $x \in X$,
we will denote the derivative of $\|\cdot\|_X$ at $x$ (if it exists) by $x^* \in X^*$,
where $X^*$ is the dual space of $X$.
For each edge $vw \in E$,
fix $\varphi_{v,w}^X := (p_v-p_w)^*$;
it follows from $p$ being a differentiable point of $f_{G,X}$ that each functional $\varphi_{v,w}^X$ exists.
We say $p$ has the \textit{graded independence property} if we can order the vertices $v_1, \ldots, v_n$ so that for each $1 < j \leq n$, the set 
\begin{align*}
    \varphi^X(G,p)_j := \left\{\varphi^X_{v_j,v_i} : 1 \leq i < j, ~ v_i v_j \in E \right\}
\end{align*}
is linearly independent; 
we shall refer to $v_n$ as the \textit{highest vertex} of $G$. 

\begin{lemma}\label{lemsmallind}
    Let $X$ be a $d$-dimensional normed space. 
    Then for each $2 \leq n \leq d+1$, there exists a realisation of the complete graph $K_n$ with the graded independence property.
\end{lemma}

\begin{proof}
    It is immediate that any realisation of $K_2$ in $X$ will have the graded independence property.
    Suppose $n>2$ and that the result holds for each complete graph with at most $n-1$ vertices.
    We shall now show the result holds for $K_n$.
    
    By our inductive assumption,
    there exists a realisation $q$ of $K_{n-1}$ which has the graded independence property with respect to some vertex ordering $v_1, \ldots, v_{n-1}$. Label the vertices of $K_n$ as $v_1, \ldots, v_{n-1}, v_n$.
    Define $A\subset X$ to be the set of points where $\|\cdot\|_X$ is differentiable and define the subset
    \begin{align*}
        B := \bigcap_{i=1}^{n-1} \{ x+ q_{v_i}: x \in A \}.
    \end{align*}
    By \cite[Theorem 25.5]{rockafellar}, $A$ is conull with respect to the Lebesgue measure  and the duality map $\psi: A\to X^*$, $\psi(x)= x^*$, is continuous. 
    It follows that $B$ is also conull, and hence a dense subset of $X$, and that the linear span of $\psi(A)=\{x^* : x\in A\}$ is $X^*$.
    Hence there exists $y \in A$ such that the functionals,
    \begin{align*}
        \varphi^X_{v_{n-1},v_1} ~, ~ \ldots ~ , ~ \varphi^X_{v_{n-1},v_{n-2}} ~,~ y^*
    \end{align*}
    are linearly independent in $X^*$.
    
    Let $(x_k)_{k \in \mathbb{N}}$ be the sequence in $X$ given by $x_k := q_{v_{n-1}} + \frac{1}{k} y$, so that    $x_k-q_{v_{n-1}} \in A$ and $(x_k-q_{v_{n-1}})^*=y^*$ for each $k \in \mathbb{N}$.
    As $B$ is dense in $X^*$,
    we may choose for each $k \in \mathbb{N}$ some element $z_k \in B$ sufficiently close to $x_k$ so we may suppose that 
    \begin{align*}
    z_k\to q_{v_{n-1}} \quad \text{ and } \quad
         \|(z_k-q_{v_{n-1}})^* - y^*\|_{X^*}=\|(z_k-q_{v_{n-1}})^* - (x_k-q_{v_{n-1}})^*\|_{X^*}         < \frac{1}{k}.
    \end{align*}
    Define the map $J :B \rightarrow L(X, \mathbb{R}^{n-1})$, where for each $z \in B$, $J(z)$ is the linear map
    \begin{align*}
    x \mapsto J(z)x := \Big((z - q_{v_1})^* (x) ~ , ~ \ldots ~ , ~ (z - q_{v_{n-1}})^*(x) \Big).
    \end{align*}
    Since $\psi$ is continuous on  $A$,
    the map $J$ is continuous.
    Now let $T \in L(X, \mathbb{R}^{n-1})$ be the linear map where for all $x \in X$,
    \begin{align*}
    T(x) := \left(\varphi^X_{v_{n-1},v_1}(x) ~, ~ \ldots ~, ~ \varphi_{v_{n-1},v_{n-2}}^X (x) ~, ~ y^*(x) \right).
    \end{align*}
    Because of our choice of $y$, the map $T$ is surjective. 
    Since the subset $S$ of surjective linear maps in $L(X, \mathbb{R}^{n-1})$ is open and $J(z_k) \rightarrow T$ as $k \rightarrow \infty$, there exists some $N \in \mathbb{N}$ where $J(z_N) \in S$. 
    
    Define $p$ to be the realisation of $K_n$ with $p_{v_i} = q_{v_i}$ for $1 \leq i \leq n-1$ and $p_{v_n} = z_N$. 
    As $J(z_N)$ is surjective,
    the set $\varphi^X(K_n,p)_n$ is linearly independent.
    For each $1 \leq j \leq n-1$,
    we have $\varphi^X(K_n,p)_j = \varphi^X(K_{n-1},q)_j$,
    hence $p$ has the graded independence property as required.
\end{proof}

\begin{lemma}\label{gip}
    Let $X$ be a finite-dimensional normed space, $G$ be a graph and $p$ be a differentiable point of the measurement map $f_{G,X}$. 
    If $p$ has the graded independence property, then $p$ is an independent realisation of $G$ in $X$.
\end{lemma}

\begin{proof}
    Suppose that $p$ is not an independent realisation of $G$ in $X$,
    i.e., $\rank df_{G,X} (p) < |E|$.
    Then there exists a non-zero map $a : E \rightarrow \mathbb{R}$ where for each $v \in V$ we have the following equality (here $w \sim v$ denotes that $vw \in E$):
    \begin{equation}\label{eq:stress}
        \sum_{w \sim v} a(vw) \varphi^X_{v,w} = 0.
    \end{equation}
    We shall prove that any map $a$ that satisfies the above conditions must be the zero map and hence obtain a contradiction.
    
     Since $p$ has the graded independence property with respect to some vertex ordering $v_1, \ldots, v_{n}$, 
    the set $\varphi^X(G,p)_n$ is linearly independent,
    and so $a(v_n w) = 0$, for every $w \sim v_n$.
    Now suppose that for some $1 \leq k < n$,
    every edge $v_i v_j \in E$ with either $i \geq k+1$ or $j \geq k+1$ gives $a(v_i v_j) = 0$.
    Then \cref{eq:stress} at vertex $v_k$ gives
    \begin{align*}
        0 = \sum_{w \sim v_k} a(v_k w) \varphi^X_{v_k , w} = \sum_{v_i \sim v_k,~ i < k} a(v_k v_i) \varphi^X_{v_k , v_i},
    \end{align*}
    and so, since the set $\varphi^X(G,p)_k$ is linearly independent, we have $a(v_k w) = 0$ for every edge $v_k w \in E$.
    By induction it follows that $a(vw) = 0$ for every edge $vw \in E$,
    contradicting our initial assumption that $a$ is a non-zero map.
\end{proof}

\begin{lemma}\label{L:IND1}
    Let $G$ be a graph with $n$ vertices and $X$ be a normed space of dimension $n-1 \leq d < \infty$.
    Then $G$ is independent in $X$.
\end{lemma}

\begin{proof}
    By \Cref{lemsmallind}, there exists a realisation $p$ of $K_n$ in $X$ with the graded independence property.
    It is immediate that $p$ is also a realisation of $G$ in $X$ with the graded independence property.
    Hence by \Cref{gip}, $G$ is independent in $X$ as required.
\end{proof}

\begin{theorem}\label{c:ind}
    Let $X$ be a finite-dimensional normed space where the norm is $C^\infty$-differentiable on an open dense subset of $X$,
    and let $Y$ be an infinite-dimensional normed space. 
    If a graph $G$ is $(X,Y)$-flattenable,
    then $G$ is independent in $X$.
\end{theorem}

\begin{proof}
    Fix $n$ to be the number of vertices of $G$. Choose any $n$-dimensional subspace $Z$ of $Y$.
    Then $G$ is $(X,Z)$-flattenable. 
    By \Cref{L:IND1},
    $G$ is independent in $Z$.
    The result now follows from applying \Cref{t:ind} to the triple $G,X,Z$.
\end{proof}

\begin{remark}
    If a graph $G=(V,E)$ is independent in $\ell_\infty^d$,
    then there exists pairwise-disjoint edge subsets $T_1,\ldots, T_d$ such that the graphs $(V,T_1), \ldots, (V,T_d)$ are forests and $E = \bigcup_{i=1}^d T_d$ (see \cite{K15} for more detail).
    Using this, it is immediate that \Cref{c:ind} gives an alternative proof of \cite[Lemma 2.5]{fhjv2017},
    namely that if $G$ is $(\ell_\infty^d, \ell_\infty)$-flattenable, then the edges of $G$ can be partitioned into $d$ edge-disjoint forests.
\end{remark}


\section{Flattening from \texorpdfstring{$\ell_p$}{lp} spaces}\label{sec:lp}

In this section we restrict to the cases where either $X$ or $Y$ are $\ell_p$ spaces.

\subsection{Bounds on flattenability}

For this subsection we prove the following result.

\begin{theorem}\label{t:linfl2}
    Let $X,Y$ be normed spaces and $G=(V,E)$ be a graph.
    \begin{enumerate}
        \item \label{t:linfl2item1} If $G$ is $(X,\ell_\infty)$-flattenable graph, then it is $(X,Y)$-flattenable.
        \item \label{t:linfl2item2} Suppose $X$ is finite-dimensional and $Y$ is infinite-dimensional.
        If $G$ is $(X,Y)$-flattenable graph, then it is $(X,\ell_2)$-flattenable.
    \end{enumerate}
\end{theorem}

\Cref{t:linfl2} shows us that so long as $\dim X < \infty = \dim Y$, the set of $(X,Y)$-flattenable graph will: (i) contain the set of $(X,\ell_\infty)$-flattenable graphs, and (ii) be contained in the set of $(X,\ell_2)$-flattenable graphs.
This result shall be important later in \Cref{sec:normedplane}.

For the proof of \Cref{t:linfl2},
we require the following three results.

\begin{theorem}[\cite{frechet10}]\label{t:embed}
	For every finite metric space $(M,d)$,
	there exists an isometry $f:M \rightarrow \ell_\infty$.
\end{theorem}

\begin{theorem}[{\cite[Theorem 1]{Shkarin}}]\label{t:shkarin}
    Let $S$ be a finite affinely independent subset of $\ell_2$.
    Then there exists $n \in \mathbb{N}$ such that for any normed space $X$ with $\dim X \geq n$,
    the set $S$ can be isometrically embedded into $X$. 
\end{theorem}

\begin{lemma}\label{l:dense}
    Let $G=(V,E)$ be any graph, $X$ be a finite-dimensional normed space and $Y$ be any normed space.
    Suppose that there exists a dense subset $D \subset Y^V$ where for each $p \in D$ there exists $q \in X^V$ where $f_{G,X}(q)=f_{G,Y}(p)$.
    Then $G$ is $(X,Y)$-flattenable.
\end{lemma}

\begin{proof}
    Without loss of generality, we may assume $G$ is connected.
    Choose any point $p \in Y^V$.
    Let $(p^n)_{n \in \mathbb{N}}$ be a sequence in $D$ that converges to $p$.
    Fix a vertex $u \in V$.
    For each $p^n$, there exists $q^n \in X^V$ such that $f_{G,X}(q^n) = f_{G,Y}(p^n)$.
    By applying translations, we may assume $q^n_u = 0$ for each $n \in \mathbb{N}$.
    As the graph $G$ is connected and $X$ is finite-dimensional,
    there exists a compact set $C \subset X^V$ such that $q^n \in C$ for all $n \in \mathbb{N}$.
    Hence there exists a convergent subsequence $(q^{n_i})_{i \in \mathbb{N}}$ with limit $q \in X^V$.
    Since both $f_{G,X}$ and $f_{G,Y}$ are continuous,
    we have 
    \begin{align*}
        f_{G,X}(q) = \lim_{i \rightarrow \infty} f_{G,X}(q^{n_i}) = \lim_{i \rightarrow \infty} f_{G,Y}(p^{n_i}) = f_{G,Y}(p)
    \end{align*}
    as required.
\end{proof}

\begin{proof}[Proof of \Cref{t:linfl2}]
	First suppose $G$ is $(X,\ell_\infty)$-flattenable.
	Choose any $q \in Y^V$.
	For each (possibly not distinct) pair $v,w \in V$,
	define $d_{vw} := \|q_v-q_w\|_Y$.
	It follows that we can define a metric space $(V,d)$ where $d(v,w):=d_{vw}$.
	By \Cref{t:embed}, the metric space $(V,d)$ can be isometrically embedded into $\ell_\infty$,
	and so there exists $p' \in \ell_\infty^V$ where $f_{G, \ell_\infty}(p') = f_{G, Y}(q)$.
	As $G$ is $(X,\ell_\infty)$-flattenable, there exists $p \in X^V$ where $f_{G,X}(p) = f_{G, \ell_\infty}(p') = f_{G, Y}(q)$ as required.
	
	Now suppose $G$ is $(X,Y)$-flattenable and $\dim X < \dim Y = \infty$.
	Choose any $p \in \ell_2^V$ so that the set $\{p_v : v \in V\}$ is affinely independent.
    By \Cref{t:shkarin},
    there exists $\widetilde{p} \in Y^V$ where $f_{G,\ell_2}(p) = f_{G,Y}(\widetilde{p})$.
    As $G$ is $(X,Y)$-flattenable,
    it follows that there exists $q \in X^V$ so that $f_{G,X}(q)= f_{G,\ell_2}(p)$.
    As the set of affinely independent realisations in $\ell_2^V$ form a dense subset,
    $G$ is $(X,\ell_2)$-flattenable by \Cref{l:dense}.
\end{proof}

\begin{remark}
    It is important to note that \Cref{t:linfl2}\ref{t:linfl2item2} requires the assumption that $Y$ is infinite-dimensional.
    For example,
    if $\dim X = \dim Y = 1$ then every graph is $(X,Y)$-flattenable by \Cref{p:1dflat}\ref{p:1dflat1}, but the only $(X,\ell_2)$-flattenable graphs are forests by \Cref{p:1dflat}\ref{p:1dflat2}.
    It is, however, unclear whether we require that $X$ is finite-dimensional.
\end{remark}

\subsection{Varying \texorpdfstring{$p$}{p}}

For a graph $G=(V,E)$ and a normed space $X$,
define the set
\begin{align*}
    \ell(G,X) := \{p \in [1,\infty] : G \text{ is $(X,\ell_p)$-flattenable} \}.
\end{align*}
It follows from \Cref{t:l1,t:lp,t:linf} that $p \in \ell(G,X)$ if and only if,
given $d \geq \binom{|V|}{2}$,
the graph $G$ is $(X,\ell_p^d)$-flattenable.

\begin{proposition}
    For every graph $G=(V,E)$ and every finite-dimensional normed space $X$,
    the set $\ell(G,X)$ is a closed subset of $[1,\infty]$.
\end{proposition}

\begin{proof}
    We may assume without loss of generality that $G$ is connected.
    Choose a value $p \in [1,\infty]$ in the closure of $\ell(G,X)$ and fix $d \geq \binom{|V|}{2}$.
    As $p$ lies in the closure of $\ell(G,X)$,
    there exists a sequence $(p_n)_{n \in \mathbb{N}}$ where $p_n\rightarrow p$ as $n\rightarrow \infty$ (if $p=\infty$, this means that $(p_n)_{n\in \mathbb{N}}$ diverges to infinity).
    Choose any realisation $q \in (\ell_p^d)^V$ of $G$.
    For each $n \in \mathbb{N}$, we can consider $q$ to be a realisation of $G$ in $\ell_{p_n}^d$ also, since each normed space has the same underlying vector space (i.e., $\mathbb{R}^d$).
    Fix a vertex $u\in V$.
    Since $G$ is connected, the set
    \begin{align*}
        S := \left\{ r \in X^V : r_u = 0, ~ \max_{vw\in E} \|r_v - r_w\|_X \leq \max_{vw\in E} \sup_{t\in[1\infty]} \|q_v- q_{w}\|_t \right\}
    \end{align*}
    is compact.
    For each $n\in \mathbb{N}$,
    there exists a realisation $r^n \in S$ such that $f_{G,X}(r^n) = f_{G,\ell_{p_n}^d}(q)$ and $r^n_v = 0$.
    As $S$ is compact,
    there exists a convergent subsequence $(r^{n_i})_{i\in \mathbb{N}}$ with limit $r \in S$.
    For each $vw \in E$,
    we have
    \begin{align*}
        \|q_v- q_{w}\|_p = \lim_{i \rightarrow \infty} \|q_v - q_w\|_{p_i} = \lim_{i \rightarrow \infty} \|r^{n_i}_v - r^{n_i}_w\|_X = \|r_v - r_w\|_X,
    \end{align*}
    hence $f_{G,X}(r) = f_{G,\ell_p^d}(q)$ and $G$ is $(X,\ell_p^d)$-flattenable.
\end{proof}

We recall that for any $p \in [1,\infty]$,
$L_p[0,1]$ is the normed space of measurable functions $f:[0,1] \rightarrow \mathbb{R}$ (modulo equality on measure 1 sets) where $\|f\|_{L_p[0,1]} < \infty$,
given the norm
\begin{align*}
    \|f\|_{L_p[0,1]} := \left(\int_0^1 |f(t)|^p dt \right)^{1/p} ~ (p < \infty), \qquad \|f\|_{L_\infty[0,1]} := \esssup_{t \in [0,1]} |f(t)|;
\end{align*}
here $\esssup_{t \in [0,1]} |f(t)|$ is the essential supremum of $f$,
i.e., the smallest value $\lambda \in \mathbb{R}$ such that $f(t) \leq \lambda$ for almost all $t\in [0,1]$.

\begin{lemma}\label{l:bigorlilp}
    For all $p\in [1,\infty]$,
    every graph is $(\ell_p, L_p[0,1])$-flattenable.
\end{lemma}

\begin{proof}
    This follows directly from the proof of \cite[Proposition 1]{ball90}.
\end{proof}

\begin{theorem}[\cite{herz}]\label{t:lessthan2}
    For all $1 \leq p \leq q \leq 2$,
    $L_q[0,1]$ isometrically embeds in $L_p[0,1]$.
\end{theorem}

With this, we can state the following.

\begin{lemma}\label{l:pqembed1}
    For all $1 \leq p \leq q \leq 2$,
    every graph is $(\ell_p, \ell_q)$-flattenable.
\end{lemma}

\begin{proof}
    Define the linear map $T: \ell_q \rightarrow L_q[0,1]$ which maps $(x_1,x_2,\ldots)$ to the function $f:[0,1] \rightarrow \mathbb{R}$ where $f(x) = 2^{1/q} x_1$ for all $x \in [0,1/2]$, and $f(x) = 2^{n/q} x_n$ if $x \in ( \sum_{i=1}^{n-1} (1/2)^i, \sum_{i=1}^n (1/2)^i]$ for some $n \geq 2$.
    Since the map $T$ is an isometry, it follows that $\ell_q$ can be embedded isometrically into $L_q[0,1]$.
    Hence, by \Cref{t:lessthan2},
    every graph is $(L_p[0,1],\ell_q)$-flattenable.
    The result now follows from \Cref{p:xyzembed} and \Cref{l:bigorlilp}.
\end{proof}

Note that the proof of \Cref{l:pqembed1} fails for $q>2$, since $\ell_q^3$ cannot be embedded in $\ell_1$ when for $q>2$ (see \cite{dor}).
It is unknown to the authors if every graph is $(\ell_p,\ell_q)$-flattenable when $q > 2$ and $p < q$.

\begin{theorem}\label{t:pqembed2}
    Let $X$ be a normed space and $1 \leq p \leq q \leq 2$.
    If $G$ is $(X,\ell_p)$-flattenable, then it is $(X,\ell_q)$-flattenable.
    Hence, the set $\ell(G,X) \cap [1,2]$ is either empty or a closed interval containing $2$.
\end{theorem}

\begin{proof}
    This follows from \Cref{p:xyzembed} and \Cref{l:pqembed1}.
\end{proof}

While \Cref{t:pqembed2} guarantees that, so long as $1 \leq p \leq q \leq 2$,  $(X,\ell_p)$-flattenability implies $(X,\ell_q)$-flattenability,
the converse is not true.
For example, while $K_4$ is $(\ell_2^3,\ell_2)$-flattenable (\Cref{t:3dflat}),
it is not $(\ell_2^3,\ell_1)$-flattenable.
This latter statement will be proven later in \Cref{sec:normedplane} (see \Cref{t:strconv}).

\section{Characterising forbidden minors for flattenability}\label{sec:normedplane}

We shall now focus on the special case of $X$ being a normed plane and $Y$ being a normed space of strictly higher dimension. Our goal is to characterise $(X,Y)$-flattenability by identifying a finite list of forbidden minors, as suggested by the celebrated Robertson–Seymour theorem \cite{rs04}.
We do so with the following result.

\begin{theorem}\label{t:2dcombined}
    Let $X$ be a normed plane and $Y$ an infinite-dimensional normed space.
    Then a graph $G$ is $(X,Y)$-flattenable if and only if either:
    \begin{enumerate}
        \item $X$ is not isometrically isomorphic to $\ell_\infty^2$ and $G$ contains no $K_4$ minor.
        \item $X$ is isometrically isomorphic to $\ell_\infty^2$ and $G$ contains no $W_4$ or $K_4 +_e K_4$ minor (see \Cref{fig:w5join}).
    \end{enumerate}
\end{theorem}

The proof of \Cref{t:2dcombined} is split into two sub-cases;
either $X$ is not isometrically isomorphic to $\ell_\infty^2$ (\Cref{t:2d1}), or it is isometrically isomorphic to $\ell_\infty^2$ (\Cref{t:l2tolinf}).
In fact,
the former proof gives a slightly stronger result,
since it only requires that $\dim Y \geq 3$.

\subsection{Flattening the complete graph of size 4}\label{sec:clique4}

In \Cref{t:2dflat}, it was shown that a graph $G$ is $(\ell_2^2,\ell_2)$-flattenable if and only if it contains no $K_4$ minor. In this subsection, we extend this result to $(X,Y)$-flattenable graphs, where $X$ is a normed plane that is not isometrically isomorphic to $\ell_\infty^2$ and $\dim Y \geq 3$.

For the next result we define the following.
A subset $S$ of a normed space $X$ is called an \emph{equilateral set} if   $\|x-y\|=1$ for all $x,y\in S$.

\begin{theorem}[\cite{petty}]\label{t:petty}
Let $X$ be a normed space.
\begin{enumerate}
\item  If $\dim X=2$ then:
\begin{enumerate}
\item[(a)] the maximal size of an equilateral set in $X$ is $3$ if and only if $X$ is not isometrically isomorphic to $\ell_\infty^2$, and
\item[(b)] the maximal size of an equilateral set in $X$ is $4$ if and only if $X$ is isometrically isomorphic to $\ell_\infty^2$.
\end{enumerate}
\item If $\dim X=3$ then there exists an equilateral set in $X$ of size $4$.
\end{enumerate}
\end{theorem}

\begin{lemma}\label{l:h1}
	Let $X,Y$ be normed spaces where $\dim X \geq 2$ and let $G$ be a $(X,Y)$-flattenable graph.
	Suppose $G'$ is formed from $G$ by adding a vertex $v_0$ and edges $v_0v_1,v_0v_2$, where $v_1,v_2 \in V$ are adjacent vertices in $G$.
	Then $G'$ is $(X,Y)$-flattenable also.
\end{lemma}

\begin{proof}
	Choose any $q' \in Y^{V \cup \{v_0\}}$.
	By \Cref{p:isomembed} we may assume that $\dim Y \geq 2$.
	Define $q \in Y^V$ to be the point with $q_v :=q'_v$ for all $v \in V$.
	As $G$ is $(X,Y)$-flattenable then there exists $p \in X^V$ where $f_{G,X}(p)=f_{G,Y}(q)$.
	Define $d_i := \|q'_{v_i}-q'_{v_0}\|_Y$ for each $i \in \{1,2\}$.
	As $\|q'_{v_1}-q'_{v_2}\|_Y = \|p_{v_1}- p_{v_2}\|_X$
	then,
	\begin{equation}\label{trianen}
		d_1 + d_2 \geq \|p_{v_1}- p_{v_2}\|_X, \quad d_1 + \|p_{v_1}- p_{v_2}\|_X \geq d_2, \quad d_2 + \|p_{v_1}- p_{v_2}\|_X \geq d_1.
	\end{equation}
	Define for each $i \in \{1,2\}$ the sets
	\begin{align*}
		D_i := \{ x \in X : \|x - p_{v_i}\|_X \leq d_{i} \}\quad \text{ and }\quad S_i := \{ x \in X : \|x - p_{v_i}\|_X = d_{i} \}.
	\end{align*}
	As \cref{trianen} holds, the set $D_1 \cap D_2$ is non-empty but does not contain $D_1$ or $D_2$.
    It follows that the set $S_1 \cap S_2$, is non-empty;
    this can be seen by traversing the boundary of $S_1$ from the point $p_{v_1} + \frac{d_1}{\|p_{v_1}-p_{v_2}\|}(p_{v_1}-p_{v_2})$ (which is not contained in $D_2$) to the point $p_{v_1} + \frac{d_1}{\|p_{v_1}-p_{v_2}\|}(p_{v_2}-p_{v_1})$ (which is contained in $D_2$) and noting that the set $S_2$ must be intersected at some point during the path.
    Hence there exists $x \in X$ where $\|p_{v_i}-x\|_X = d_i$ for each $i \in \{1,2\}$.
	If we set $p' \in X^{V \cup \{v_0\}}$ to be the realisation where $p'_v =p_v$ for all $v \in V$ and $p'_{v_0} =x$,
	then $f_{G',X}(p')=f_{G',Y}(q')$ as required.
\end{proof}

\begin{proposition}[{\cite[Proposition 7.3.1]{diestel}}]\label{p:h1}
	For any graph $G=(V,E)$,
	the following are equivalent:
	\begin{enumerate}
		\item $G$ does not contain $K_4$ as a minor,
		and $G +vw$ does contain $K_4$ as a minor for every distinct pair $v,w\in V$ where $vw \notin E$.
		\item $G$ can be formed from $K_2$ by a sequence moves where we add a vertex connected to pairs of adjacent vertices.
	\end{enumerate}
\end{proposition}

\begin{lemma}\label{l:2d2}
	Let $X,Y$ be normed spaces where $\dim X \geq 2$.
	If $G$ contains no $K_4$ minor then $G$ is $(X,Y)$-flattenable.
\end{lemma}

\begin{proof}	
	Since $(X,Y)$-flattenability is a minor-closed property,
	we may assume that $G$ is maximal in the set of $K_4$ minor-free graphs, i.e., for any $e \notin E$, the graph $G +e$ contains a copy of $K_4$ as a minor.
	By \Cref{p:h1} and \Cref{l:h1},
	$G$ is $(X,Y)$-flattenable,
	as $K_2$ is clearly $(X,Y)$-flattenable.
\end{proof}

With this, we can now prove our first key result of the section.

\begin{theorem}\label{t:2d1}
	Let $X$ be a normed plane not isometrically isomorphic to $\ell_\infty^2$,
	and let $Y$ be a normed space with $\dim Y \geq 3$.
	Then $G$ is $(X, Y)$-flattenable if and only if $G$ contains no $K_4$ minor.
\end{theorem}

\begin{proof}	
    Suppose $G$ is $(X,Y)$-flattenable and contains $K_4$ as a minor.
	Since $(X,Y)$-flattenability is a minor-closed property,
	it suffices to assume that $G=K_4$.
	By \Cref{t:petty}, there exists $q \in Y^V$ where $f_{G,Y}(q) = (1,1,1,1)$.
	As $G$ is $(X,Y)$-flattenable then there exists $x_1, x_2, x_3, x_4 \in X$ so that if we set $p_{v_i}=x_i$ for each $i \in \{1,2,3,4\}$ then $f_{G,X}(p)=(1,1,1,1)$.
	However this contradicts \Cref{t:petty}.
 
 The converse follows from \Cref{l:2d2}.
\end{proof}

\Cref{t:2d1} can be extended to higher dimensional normed spaces in certain specific cases.
For next result we first recall that a normed space $X$ is \emph{strictly convex} if for all points $x ,y \in X$ with $\|x\|=\|y\|=1$,
we have $\|x+y\|=2$ if and only if $x=y$.

\begin{theorem}\label{t:strconv}
	Let $X$ be a strictly convex normed space with $\dim X \geq 2$ and let $Y$ be a normed space that is not strictly convex.
	Then a graph $G$ is $(X,Y)$-flattenable if and only if it contains no $K_4$ minor.
\end{theorem}

\begin{proof}
	By \Cref{l:2d2},
	if $G$ contains no $K_4$ minor then it is $(X,Y)$-flattenable.
	Suppose $G$ contains $K_4$ as a minor.
	Since $(X,Y)$-flattenability is a minor-closed property,
	it suffices to assume that $G=K_4$.
	As $Y$ is not strictly convex,
	there exists $x,y \in Y$ where $\|x\|=\|y\|=1$, $x \neq y$ and $\|x +y\| = 2$.
	Let $q$ be the realisation of $K_4$ in $Y$ with
	\begin{align*}
		q_{v_1} = 0, \quad q_{v_2} = x, \quad q_{v_3} = y, \quad q_{v_4} = x+y.
	\end{align*}
	As $X$ is strictly convex then no such $p \in X^V$ exists with $f_{G,X}(p)=f_{G,Y}(q)$,
	hence $G$ is not $(X,Y)$-flattenable as required.
\end{proof}

\begin{example}\label{ex:lp}
    By \Cref{t:strconv}, the $(\ell_2^2, \ell_\infty^2)$-flattenable graphs are exactly those that contain no $K_4$ minor.
    \Cref{t:strconv} can also be applied to higher dimensions for some interesting results.
    The graph $K_4$ is $(\ell_2^3,\ell_2)$-flattenable by \Cref{t:3dflat},
    but is not $(\ell_2^3,\ell_\infty)$-flattenable by \Cref{t:strconv}.
    Hence the set of $(\ell_2^3,\ell_2)$-flattenable graphs is not equal to the set of $(\ell_2^3,\ell_\infty)$-flattenable graphs.
    As the forbidden minors of both $(\ell_2^3,\ell_2)$- and $(\ell_2^3,\ell_\infty)$-flattenability are known,
    we can then use \Cref{t:linfl2} to give ``upper and lower bounds'' on the $(\ell_2^3,Y)$-flattenable graphs when $\dim Y = \infty$.
\end{example}

\subsection{Flattening into \texorpdfstring{$\ell_\infty^2$}{2D L-infinity}}

Let $M$ be a real $n \times n$ symmetric matrix.
We say that $M$ is a \emph{Euclidean distance matrix} if there exists a map $p :\{ 1,\ldots,n\} \rightarrow \ell_2$ where $\|p_i - p_j\|^2_2 = M_{i,j}$ for all $i, j \in \{ 1,\ldots,n\}$.

\begin{theorem}[{\cite[Theorem 1]{schoenberg1935remarks}}]\label{t:schoen}
    Let $M$ be a real $n \times n$ symmetric matrix with $M_{ii} = 0$ for all $i \in \{1,\ldots,n\}$.
    Define $A$ to be the $(n-1)\times (n-1)$ symmetric matrix with $A_{ij} = M_{in} + M_{jn} - M_{ij}$ for all $i,j \in \{1,\ldots,n-1\}$.
    Then $M$ is a Euclidean distance matrix if and only if $A$ is positive semidefinite.
\end{theorem}


For the next result we recall that a \emph{(symmetric) partial matrix} is a real-valued indexed tuple $\widetilde{M} = (m_{ij})_{(i,j) \in E}$ for some indexing set $E \subset \{(i,j) : 1 \leq i \leq j \leq n\}$.
A \emph{completion} of a partial matrix $\widetilde{M}$ is a symmetric matrix $M$ where $M_{ij}=m_{ij}$ for each $(i,j) \in E$.

\begin{theorem}\label{t:l2tolinf}
    Let $X$ be a normed plane isometrically isomorphic to $\ell_\infty^2$,
    and let $Y$ be an infinite-dimensional normed space.
    Then a graph is $(X,Y)$-flattenable if and only if it contains no $W_4$ or $K_4+_e K_4$ minor (see \Cref{fig:w5join}).
\end{theorem}

\begin{proof}
By  \Cref{p:isom}, we may assume $X= \ell_\infty^2$.
    By \Cref{t:linfl2},
    for the forward implication it suffices to consider the case where   $Y  = \ell_2$.
    
    The matrix $M(W_4)$ below is a completion of the partial matrix formed from the squares of the edge lengths of $W_4$ given in \Cref{fig:w5joinedgelengths}, where bold numbers represent entries added to complete the partial matrix:
    \begin{align*}
        M(W_4) :=
        \begin{bmatrix}
            0 & 324 &  \mathbf{245} & 576 & 40000 \\
            324 & 0 & 289 & \mathbf{294} &  40000\\
            \mathbf{245} & 289 & 0 & 400 & 40000 \\
            576 & \mathbf{294} & 400 & 0 & 40000 \\
            40000 & 40000 & 40000 & 40000 & 0
        \end{bmatrix}.
    \end{align*}
    Similarly,
    the matrix $M(K_4+_e K_4)$ is a completion of the partial matrix formed from the squares of the edge lengths of $K_4+_e K_4$ given in \Cref{fig:w5joinedgelengths},
    where bold numbers represent the added values:
    \begin{align*}
        M(K_4+_e K_4) :=
        \begin{bmatrix}
            0 & \mathbf{3003} & 5041 & 5929 & 5476 & 2116 \\
            \mathbf{3003} & 0 & 2809 & 7744 & 6241 & 1296 \\
            5041 & 2809 & 0 & 6084 & \mathbf{4765} & \mathbf{2595} \\
            5929 & 7744 & 6084 & 0 & \mathbf{6545} & \mathbf{4655} \\
            5476 & 6241 & \mathbf{4765} & \mathbf{6545} & 0 & 6241 \\
            2116 & 1296 & \mathbf{2595} & \mathbf{4655} & 6241 & 0
        \end{bmatrix}.
    \end{align*}
    By applying \Cref{t:schoen} combined with a simple computational eigenvalue check,
    we see that both $M(W_4)$ and $M(K_4+_e K_4)$ are Euclidean distance matrices.
    Hence there exist realisations $p$ and $q$ of $W_4$ and $K_4+_e K_4$ respectively in $\ell_2$ that realise the edge lengths shown in \Cref{fig:w5joinedgelengths}.
    It was shown in \cite{fhjv2017} that there exist no realisations of either graph in $\ell_\infty^2$ that satisfy the given edge lengths in \Cref{fig:w5joinedgelengths}.
    Hence $W_4$ and $K_4+_e K_4$ are not $(\ell_\infty^2,\ell_2)$-flattenable,
    and so any $(\ell_\infty^2,\ell_2)$-flattenable graph must contain no $W_4$ or $K_4+_e K_4$ minor.
    
   For the converse, any graph that contains no $W_4$ or $K_4+_e K_4$ minor is $(\ell_\infty^2,\ell_\infty)$-flattenable by \Cref{t:infflat}.
    Hence any graph that contains no $W_4$ or $K_4+_e K_4$ minor is $(\ell_\infty^2,Y)$-flattenable by \Cref{t:linfl2}.
\end{proof}

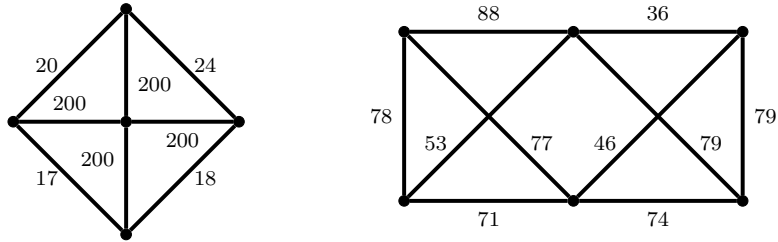
\begin{figure}[htp]
\begin{center}
    \begin{tikzpicture}[scale=1.5]
        \node[vertex] (a) at (2,0) {};
    	\node[vertex] (b) at (1,-1) {};
        \node[vertex] (c) at (1,1) {};
    	\node[vertex] (d) at (0,0) {};
    	
        \node[vertex] (e) at (1,0) {};
    	
    	\draw[edge] (e)to node[below,labelsty] {200} (a);
    	\draw[edge] (e)to node[above left,labelsty] {200} (b);
    	\draw[edge] (e)to node[below right,labelsty] {200} (c);
    	\draw[edge] (e)to node[above,labelsty] {200} (d);
    	\draw[edge] (a)to node[right,labelsty] {18} (b);
    	\draw[edge] (a)to node[right,labelsty] {24} (c);
    	\draw[edge] (b)to node[left,labelsty] {17} (d);
    	\draw[edge] (c)to node[left,labelsty] {20} (d);
    	
    	\filldraw[black] (0,-1.5) circle {};
    \end{tikzpicture}\qquad \qquad
    \begin{tikzpicture}[scale=1.125]
        \node[vertex] (a) at (-2,-1) {};
    	\node[vertex] (b) at (-2,1) {};
    	
        \node[vertex] (c) at (0,1) {};
    	\node[vertex] (d) at (0,-1) {};
    	
        \node[vertex] (e) at (2,-1) {};
    	\node[vertex] (f) at (2,1) {};
    	
    	\node[blankvertex] (blank) at (0,-2) {};
    	
    	\draw[edge] (a)to node[left,labelsty] {78} (b);
    	\draw[edge] (a)to node[xshift=-20pt,yshift=-10pt,labelsty] {53} (c);
    	\draw[edge] (a)to node[below,labelsty] {71} (d);
    	\draw[edge] (b)to node[above,labelsty] {88} (c);
    	\draw[edge] (b)to node[xshift=20pt,yshift=-10pt,labelsty] {77} (d);
        
    	\draw[edge] (c)to node[xshift=20pt,yshift=-10pt,labelsty] {79} (e);
    	\draw[edge] (c)to node[above,labelsty] {36}(f);
    	\draw[edge] (d)to node[below,labelsty] {74} (e);
    	\draw[edge] (d)to node[xshift=-20pt,yshift=-10pt,labelsty] {46} (f);
    	\draw[edge] (e)to node[right,labelsty] {79}(f);
    \end{tikzpicture}
\end{center}
\vspace{-0.5cm}
\caption{The edge lengths assigned to $W_4$ (left) and $K_4 +_e K_4$ (right).}
\label{fig:w5joinedgelengths}\textit{}
\end{figure}

\begin{remark}\label{rem:sdp2}
    The matrices $M(W_4)$ and $M(K_4 +_e K_4)$ in the proof of \Cref{t:l2tolinf} were obtained as follows. We first applied a method similar to what is sketched out in \cite{7298562} using a semidefinite program solver in Julia \cite{convexjl} to obtain completed matrices.
    After this, we then rounded the entries of each matrix to obtain $M(W_4)$ and $M(K_4 +_e K_4)$.
    We next checked that both $M(W_4)$ and $M(K_4 +_e K_4)$ were Euclidean distance matrices by computing each one's unique $A$ matrix (as described in \Cref{t:schoen})  and checking if it were positive semidefinite.
    This last step was performed by computing the exact eigenvalues of each $A$ matrix.
    %
\end{remark}

\subsection{Flattening into \texorpdfstring{$\ell_2^2$}{2D L2} from any normed space}

\Cref{p:isomembed} and \Cref{t:2d1} characterise exactly which graphs are $(\ell_2^2,Y)$-flattenable when $\dim Y \neq 2$;
i.e., all graphs if $\dim Y=1$, $K_4$ minor-free graphs if $\dim Y \geq 3$.
We also know from \Cref{t:strconv} that, when $Y$ is a normed plane that is not strictly convex,
a graph is $(\ell_2^2,Y)$-flattenable if and only if it contains no $K_4$ minor.
We now improve this latter result by dropping the requirement that $Y$ is not strictly convex, and hence obtain a full characterisation for the $(\ell_2^2,Y)$-flattenable graphs for any choice of $Y$.
We begin with the following two slight adaptations of results of Norlander and of Alonso and Ben\'{i}tez respectively.

\begin{theorem}[{\cite[Theorem, pg.~15]{nor}}]\label{t:norlander}
    Let $Y$ be a normed space and $0 < \varepsilon < 2$.
    Given the closed non-empty set
    \begin{equation*}
        D(\varepsilon ) := \Big\{ \| a + b\|_Y : \| a - b\|_Y = \varepsilon , ~ \| a\|_Y = \| b\|_Y = 1 \Big\} ,
    \end{equation*}
    the following inequality holds:
    \begin{equation*}
        \inf D(\varepsilon) \leq \sqrt{4 -\varepsilon^2} \leq \sup D(\varepsilon).
    \end{equation*}
\end{theorem}

\begin{remark}
    Although Norlander originally only stated the second inequality given in \Cref{t:norlander}, the method used to prove the result also gives the first inequality.
\end{remark}

\begin{corollary}[{\cite[Corollary, pg.~323]{ab88}}]\label{c:ab}
    Let $Y$ be a normed plane and define $D(\varepsilon)$ as in \Cref{t:norlander}.
    If $Y$ is not isometrically isomorphic to $\ell_2^2$, then for each $0 < \varepsilon < 2$ which is not an element of the set
    \begin{equation*}
        S := \Big\{ 2 \cos(k\pi /2n) : n,k\in \mathbb{N},~ 1\leq k\leq n \Big\},
    \end{equation*}
    we have $\inf D(\varepsilon) < \sup D(\varepsilon)$.
\end{corollary}

\begin{theorem}\label{t:ell2}
	Let $Y$ be any normed space.
    \begin{enumerate}
        \item If $\dim Y = 1$ or $Y$ is a normed plane that is isometrically isomorphic to $\ell_2^2$,
        then every graph is $(\ell_2^2,Y)$-flattenable.
        \item If $\dim Y \geq 3$ or $Y$ is a normed plane that is not isometrically isomorphic to $\ell_2^2$,
        then a graph is $(\ell_2^2,Y)$-flattenable if and only if it contains no $K_4$ minor.
    \end{enumerate}
\end{theorem}

\begin{proof}
	As noted above, the remaining case to check is exactly when $Y$ is a normed plane that is not isometrically isomorphic to $\ell_2^2$.
    By \Cref{l:2d2},
	if $G$ contains no $K_4$ minor then it is $(\ell_2^2,Y)$-flattenable.
	Suppose $G$ contains $K_4$ as a minor.
	Since $(\ell_2^2,Y)$-flattenability is a minor-closed property,
	it suffices to assume that $G\cong K_4$.
	Label the vertices of $G$ by $v_1,v_2,v_3,v_4$.
	
	Choose $\varepsilon \in (0,2) \setminus S$,
	with $S$ being the set defined in \Cref{c:ab}.
	It is easy to see that, up to isometry,
	there exist exactly two realisations of $K_4$ in $\ell_2^2$ such that the edges $v_1v_2,v_1v_3,v_2v_4,v_3v_4$ have length 1 and the edge $v_2 v_3$ has length $\varepsilon$:
	the realisation $p$ where $\|p_{v_1} - p_{v_4}\|_Y = \sqrt{4 - \varepsilon^2}$,
	and the realisation $p'$ where $\|p'_{v_1} - p'_{v_4}\|_Y = 0$.
	By \Cref{t:norlander} and \Cref{c:ab},
	there exists $a,b \in Y$ such that $\|a\|_Y=\|b\|_Y = 1$, $\|a-b\|_Y = \varepsilon$ and $\|a+b\|_Y \neq \sqrt{4 - \varepsilon^2}$.
	Since $\|a-b\|_Y < 2$,
	we also have that $\|a+b\|_Y \neq 0$.
	Define $q$ to be the realisation of $K_4$ in $Y$ with $q_{v_1}=0$, $q_{v_2} = a$, $q_{v_3} = b$ and $q_{v_4} = a+b$.
	Since 
	\begin{align*}
	    \|q_{v_1} - q_{v_2}\|_Y = \|q_{v_1} - q_{v_3}\|_Y = \|q_{v_2} - q_{v_4}\|_Y = \|q_{v_3} - q_{v_4}\|_Y = 1
	\end{align*}
	and $\|q_{v_2} - q_{v_3}\|_Y = \|a-b\|_Y = \varepsilon$,
	but $\|q_{v_1}-q_{v_4}\|_Y = \|a+b\|_Y \notin \{0,\sqrt{4 - \varepsilon^2}\}$,
	we have that $f_{G,Y}(q) \notin f_{G,\ell_2^2}((\ell_2^2)^V)$,
	i.e., $G$ is not $(\ell_2^2,Y)$-flattenable.
\end{proof}

\section{Flattenability for countably infinite graphs}\label{sec:infinite}

For this final section we shall now allow a graph $G=(V,E)$ to have countably infinite vertex and edge sets.
Our definitions of flattenability extend immediately to countably infinite graphs.


\begin{definition}
	A \emph{tower} in $G$ is a sequence of finite subgraphs $(G_n)_{n \in \mathbb{N}}$ with $G_n=(V_n,E_n)$,
	where $V_n \subset V_{n+1}$ and $E_n \subset E_{n+1}$ for each $n \in \mathbb{N}$.
	A tower is \emph{complete} if $\bigcup_{n \in \mathbb{N}} V_n = V$ and $\bigcup_{n \in \mathbb{N}} E_n = E$.
\end{definition}
The following lemma is a well-known application of Tychonoff's theorem. We provide the proof for completeness.
\begin{lemma}\label{l:inv}
	Let $(A_n)_{n \in \mathbb{N}}$ be a sequence of non-empty compact Hausdorff spaces where for each $n \leq m$ there exists a continuous map $\pi_{n,m} :A_m \rightarrow A_n$.
	Suppose that $\pi_{n,n}$ is the identity map and $\pi_{n,\ell} = \pi_{n,m} \circ \pi_{m ,\ell}$ for all $n \leq m \leq \ell$.
	Then the inverse limit
	\begin{align*}
		A := \left\{ (a_n)_{n \in \mathbb{N}} \in \prod_{n \in \mathbb{N}} A_n : \pi_{n,m}(a_m)=a_n \text{ for all } n \leq m \right\}
	\end{align*}
	is a non-empty compact subset of $\prod_{n \in \mathbb{N}} A_n$.
\end{lemma}

\begin{proof}
	By Tychonoff's theorem,
	$\prod_{n \in \mathbb{N}} A_n$ is compact;
	further,
	since the product of Hausdorff spaces is Hausdorff then $\prod_{n \in \mathbb{N}} A_n$ is Hausdorff also.
	Since $A$ is a closed subset of $\prod_{n \in \mathbb{N}} A_n$ then it is a compact subset.
	
	For any $n \leq m$ we note that $\pi_{1,n}(A_n) \supset \pi_{1,m}(A_m)$.
	As each $\pi_{1,n}(A_n)$ is a non-empty compact subset of a Hausdorff space then there exists $x \in \cap_{n \in \mathbb{N}} \pi_{1,n}(A_n)$.	
	Define for each $k \in \mathbb{N}$ the set
	\begin{align*}
		B_k := \left\{ (a_n)_{n \in \mathbb{N}} \in \prod_{n \in \mathbb{N}} A_n : a_1=x,~ \pi_{n,k}(a_k)=a_n \text{ for all } n \leq k \right\}.
	\end{align*}
	By our choice of $x \in A_1$,
	each $B_k$ is a non-empty compact subset of $\prod_{n \in \mathbb{N}} A_n$,
	and we note that $B_k \supset B_\ell$ for all $k \leq \ell$.
 By Cantor's intersection theorem,  $A = \bigcap_{n \in \mathbb{N}} B_n$ is a non-empty set.
\end{proof}

\begin{theorem}\label{t:infinite}
	Let $G =(V,E)$ be a connected graph with countable vertex set.
	Then the following are equivalent for any normed spaces $X,Y$ where $X$ is finite-dimensional:
	\begin{enumerate}
		\item\label{t:infinite1} $G$ is $(X,Y)$-flattenable.
		\item\label{t:infinite2} Every subgraph of $G$ is $(X,Y)$-flattenable.
		\item\label{t:infinite3} $G$ contains a complete tower of connected $(X,Y)$-flattenable subgraphs.
	\end{enumerate}
\end{theorem}

\begin{proof}
	It is immediate that \ref{t:infinite1} $\Rightarrow$ \ref{t:infinite2} $\Rightarrow$ \ref{t:infinite3}.
	Suppose \ref{t:infinite3} holds,
	i.e.~there exists a complete tower $(G_n)_{n \in \mathbb{N}}$ where each $G_n$ is connected and $(X,Y)$-flattenable.
	Choose any $q \in Y^V$ and $v_0 \in V_1$.
	For each $n \in \mathbb{N}$,
	let 
	\begin{align*}
		A_n := \left\{ p \in X^{V_n} : p_{v_0}=0,~ f_{G_n,X}(p) = f_{G_n,Y}(q|_{V_n}) \right\}.
	\end{align*}
	As each $G_n$ is connected, $X$ is finite-dimensional and each $f_{G_n,X}$ is continuous,
	each set $A_n$ is a non-empty compact Hausdorff space.	
	For every $n\leq m$,
	define the continuous map
	\begin{align*}
		\pi_{n,m} :A_m \rightarrow A_n, ~ (p_v)_{v \in V_m} \mapsto (p_v)_{v \in V_n}.
	\end{align*}
	It is immediate that $\pi_{n,n}$ is the identity map and $\pi_{n,\ell} = \pi_{n,m} \circ \pi_{m ,\ell}$ for all $n \leq m \leq \ell$.
	Then by \Cref{l:inv} there exists $(p_n)_{n \in \mathbb{N}} \in \prod_{n \in \mathbb{N}} A_n$ where $\pi_{n,m}(p_m)=p_n$ for all $n \leq m$.
	If we define $p \in X^V$ to be the unique point where $p_v := (p_n)_v$ for $v \in V_n$ then $f_{G,X}(p)=f_{G,Y}(q)$.
	Hence $G$ is $(X,Y)$-flattenable as required.
\end{proof}

\begin{remark}
    It is worth noting that \Cref{t:infinite} requires that $X$ is finite-dimensional.
    To see why this is required,
    take $G$ to be the complete graph with a countably infinite set of vertices,
    $X = \ell_p$ for some $1 \leq p < 2$ with $p \neq 2$, and $Y= \ell_2$.
    By \Cref{t:linfl2}, every finite subgraph is $(\ell_p,\ell_2)$-flattenable.
    Suppose for contradiction the graph $G$ is also $(\ell_p,\ell_2)$-flattenable.
    Choose a realisation $q$ of $G$ such that the set $D := \{q_v : v \in V \}$ is a dense subset of $\ell_2$.
    By our assumption,
    there exists an equivalent realisation $\widetilde{q}$ in $\ell_p$.
    Hence there exists an isometry $f: D \rightarrow \ell_p$ with $f(0)= 0$.
    As $\ell_p$ is complete, we can extend this map to an isometry $h: \ell_2 \rightarrow \ell_p$.
    However,
    no such isometry can exist (see for example \cite[Section 4]{ab15}),
    which gives the desired contradiction.
\end{remark}

\section*{Acknowledgement}

This project was progressed during the Fields Institute Thematic Program on Geometric Constraint Systems, Framework Rigidity, and Distance Geometry. The authors are grateful to the Fields Institute for their hospitality and financial support.

S.\,D.\ was supported by the Heilbronn Institute for Mathematical Research.
E.\,K.\ and D.\,K.\ were partially supported by the Engineering and Physical Sciences Research Council [grant number EP/S00940X/1].
W.\,S. was partially supported by NSF DMS 1564480 and NSF DMS 1563234.

\bibliographystyle{plainurl}
\bibliography{ref}

\begin{thebibliography}{10}

\bibitem{ab15}
Fernando Albiac and Florent Baudier.
\newblock Embeddability of snowflaked metrics with applications to the
  nonlinear geometry of the spaces {$L_p$} and $\ell_p$ for $0<p<\infty$.
\newblock {\em Journal of Geometric Analysis}, 25:1--24, 2015.
\newblock \href {https://doi.org/10.1007/s12220-013-9390-0}
  {\path{doi:10.1007/s12220-013-9390-0}}.

\bibitem{ab88}
Javier Alonso and Carlos Benítez.
\newblock Some characteristic and non-characteristic properties of inner
  product spaces.
\newblock {\em Journal of Approximation Theory}, 55(3):318--325, 1988.
\newblock \href {https://doi.org/10.1016/0021-9045(88)90098-6}
  {\path{doi:10.1016/0021-9045(88)90098-6}}.

\bibitem{ball90}
Keith Ball.
\newblock Isometric embedding in $l_p$-spaces.
\newblock {\em European Journal of Combinatorics}, 11(4):305--311, 1990.
\newblock \href {https://doi.org/10.1016/S0195-6698(13)80131-X}
  {\path{doi:10.1016/S0195-6698(13)80131-X}}.

\bibitem{belkconn07}
Maria Belk and Robert Connelly.
\newblock Realizability of graphs.
\newblock {\em Discrete and Computational Geometry}, 37:125--137, 2007.
\newblock \href {https://doi.org/10.1007/s00454-006-1284-5}
  {\path{doi:10.1007/s00454-006-1284-5}}.

\bibitem{btv}
Geoffrey~J. Butler, J.~G. Timourian, and C.~Viger.
\newblock The rank theorem for locally {L}ipschitz continuous functions.
\newblock {\em Canadian Mathematical Bulletin}, 31(2):217–226, 1988.
\newblock \href {https://doi.org/10.4153/CMB-1988-034-8}
  {\path{doi:10.4153/CMB-1988-034-8}}.

\bibitem{clarke}
Frank~H. Clarke.
\newblock {\em Optimization and Nonsmooth Analysis}.
\newblock Society for Industrial and Applied Mathematics, 1990.
\newblock \href {https://doi.org/10.1137/1.9781611971309}
  {\path{doi:10.1137/1.9781611971309}}.

\bibitem{diestel}
Reihard Diestel.
\newblock {\em Graph Theory: 5th edition}.
\newblock Graduate Texts in Mathematics. Springer-Verlag, 2017.
\newblock \href {https://doi.org/10.1007/978-3-662-53622-3}
  {\path{doi:10.1007/978-3-662-53622-3}}.

\bibitem{7298562}
Ivan Dokmanic, Reza Parhizkar, Juri Ranieri, and Martin Vetterli.
\newblock Euclidean distance matrices: Essential theory, algorithms, and
  applications.
\newblock {\em IEEE Signal Processing Magazine}, 32(6):12--30, 2015.
\newblock \href {https://doi.org/10.1109/MSP.2015.2398954}
  {\path{doi:10.1109/MSP.2015.2398954}}.

\bibitem{dor}
Leonard~E. Dor.
\newblock Potentials and isometric embeddings in {$L_1$}.
\newblock {\em Israel Journal of Mathematics}, 24:260--268, 1976.
\newblock \href {https://doi.org/10.1007/BF02834756}
  {\path{doi:10.1007/BF02834756}}.

\bibitem{fhjm}
Samuel Fiorini, Tony Huynh, Gwena\"{e}l Joret, and Carole Muller.
\newblock Unavoidable minors for graphs with large
  \texorpdfstring{$\ell_p$}{lp}-dimension.
\newblock {\em Discrete and Computational Geometry}, 66:301--343, 2021.
\newblock \href {https://doi.org/h10.1007/s00454-021-00285-5}
  {\path{doi:h10.1007/s00454-021-00285-5}}.

\bibitem{fhjv2017}
Samuel Fiorini, Tony Huynh, Gwena\"{e}l Joret, and Antonios Varvitsiotis.
\newblock The excluded minors for isometric realizability in the plane.
\newblock {\em SIAM Journal on Discrete Mathematics}, 31(1):438--453, 2017.
\newblock \href {https://doi.org/10.1137/16M1064775}
  {\path{doi:10.1137/16M1064775}}.

\bibitem{frechet10}
Maurice Fr\'{e}chet.
\newblock Les dimensions d’un ensemble abstrait.
\newblock {\em Mathematische Annalen}, 68:145--168, 1910.
\newblock \href {https://doi.org/10.1007/BF01474158}
  {\path{doi:10.1007/BF01474158}}.

\bibitem{herz}
Carl~S. Herz.
\newblock A class of negative-definite functions.
\newblock {\em Proceedings of the American Mathematical Society},
  14(4):670--676, 1963.
\newblock \href {https://doi.org/10.2307/2034298} {\path{doi:10.2307/2034298}}.

\bibitem{hol}
Włodzimierz Holsztynski.
\newblock $\mathbb{R}^n$ as a universal metric space.
\newblock In {\em Notices of the AMS}, volume~25, 1978.

\bibitem{K15}
Derek Kitson.
\newblock Finite and infinitesimal rigidity with polyhedral norms.
\newblock {\em Discrete and Computational Geometry}, 54:390--411, 2015.
\newblock \href {https://doi.org/10.1007/s00454-015-9706-x}
  {\path{doi:10.1007/s00454-015-9706-x}}.

\bibitem{nor}
G{\"o}te Nordlander.
\newblock {The modulus of convexity in normed linear spaces}.
\newblock {\em Arkiv f\"{o}r Matematik}, 4(1):15--17, 1960.
\newblock \href {https://doi.org/10.1007/BF02591317}
  {\path{doi:10.1007/BF02591317}}.

\bibitem{petty}
Clinton~M. Petty.
\newblock Equilateral sets in {M}inkowski spaces.
\newblock {\em Proceedings of the American Mathematical Society},
  29(2):369--374, 1971.

\bibitem{rs95}
Neil Robertson and P.D. Seymour.
\newblock Graph minors {XIII.} {T}he disjoint paths problem.
\newblock {\em Journal of Combinatorial Theory, Series B}, 63(1):65--110, 1995.
\newblock \href {https://doi.org/10.1006/jctb.1995.1006}
  {\path{doi:10.1006/jctb.1995.1006}}.

\bibitem{rs04}
Neil Robertson and P.D. Seymour.
\newblock Graph minors {XX.} {W}agner's conjecture.
\newblock {\em Journal of Combinatorial Theory, Series B}, 92(2):325--357,
  2004.
\newblock \href {https://doi.org/10.1016/j.jctb.2004.08.001}
  {\path{doi:10.1016/j.jctb.2004.08.001}}.

\bibitem{rockafellar}
Ralph~Tyrell Rockafellar.
\newblock {\em Convex Analysis}.
\newblock Princeton University Press, Princeton, 1970.
\newblock \href {https://doi.org/doi:10.1515/9781400873173}
  {\path{doi:doi:10.1515/9781400873173}}.

\bibitem{rodl}
Vojtech R\"{o}dl and Andrzej Ruci\'{n}ski.
\newblock Bipartite coverings of graphs.
\newblock {\em Combinatorics, Probability and Computing}, 6(3):349–352, 1997.
\newblock \href {https://doi.org/10.1017/S0963548397003064}
  {\path{doi:10.1017/S0963548397003064}}.

\bibitem{schoenberg1935remarks}
Isaac~J Schoenberg.
\newblock Remarks to {M}aurice {F}r\'{e}chet's article ``sur la definition
  axiomatique d'une classe d'espace distances vectoriellement applicable sur
  l'espace de {H}ilbert''.
\newblock {\em Annals of Mathematics}, pages 724--732, 1935.
\newblock \href {https://doi.org/10.2307/1968654} {\path{doi:10.2307/1968654}}.

\bibitem{Shkarin}
Stanislav~A Shkarin.
\newblock Isometric embedding of finite ultrametric spaces in {B}anach spaces.
\newblock {\em Topology and its Applications}, 142(1):13--17, 2004.
\newblock \href {https://doi.org/10.1016/j.topol.2003.12.002}
  {\path{doi:10.1016/j.topol.2003.12.002}}.

\bibitem{SithWill}
Meera Sitharam and Joel Willoughby.
\newblock On flattenability of graphs.
\newblock In Francisco Botana and Pedro Quaresma, editors, {\em Automated
  Deduction in Geometry}, pages 129--148, Cham, 2015. Springer International
  Publishing.

\bibitem{convexjl}
Madeleine Udell, Karanveer Mohan, David Zeng, Jenny Hong, Steven Diamond, and
  Stephen Boyd.
\newblock Convex optimization in {J}ulia.
\newblock {\em SC14 Workshop on High Performance Technical Computing in Dynamic
  Languages}, 2014.
\newblock \href {http://arxiv.org/abs/1410.4821} {\path{arXiv:1410.4821}}.

\bibitem{witt}
Hans~S. Witsenhausen.
\newblock Minimum dimension embedding of finite metric spaces.
\newblock {\em Journal of Combinatorial Theory, Series A}, 42(2):184--199,
  1986.
\newblock \href {https://doi.org/10.1016/0097-3165(86)90089-0}
  {\path{doi:10.1016/0097-3165(86)90089-0}}.

\end{thebibliography}

\end{document}